\documentclass[article, 12pt]{amsart}
\title{Universal minimal flows of generalized Wa\.zewski  dendrites}

\usepackage{latexsym}
\usepackage{amssymb, amsmath, amsthm}
\usepackage{graphicx}
\usepackage{subfigure}
\usepackage{xcolor}
\usepackage{url}

\newtheorem{theorem}{Theorem}
\newtheorem{lemma}[theorem]{Lemma}
\theoremstyle{remark}
\newtheorem{remark}[theorem]{\bf Remark}
\newtheorem{corollary}[theorem]{\bf Corollary}
\newtheorem{example}[theorem]{\bf Example}
\newtheorem{proposition}[theorem]{\bf Proposition}
\newtheorem{question}[theorem]{\bf Question}

\def\acts{\curvearrowright}

\newcommand{\f}{\mathcal{F}}
\newcommand{\g}{\mathcal{G}}

\newcommand{\G}{\mathbb{G}}
\newcommand{\T}{\mathcal{T}}

\newcommand{\aut}{{\rm Aut}}
\newcommand{\lan}{\mathcal{L}}

\newcommand{\mpp}{{M_P}}
\newcommand{\A}{A}

\newcommand{\dom}{W_\omega}
\newcommand{\rest}{\restriction}
\def\aut{{\rm Aut}}

\usepackage{scalerel,stackengine}
\stackMath
\newcommand\reallywidehat[1]{%
\savestack{\tmpbox}{\stretchto{%
  \scaleto{%
    \scalerel*[\widthof{\ensuremath{#1}}]{\kern-.6pt\bigwedge\kern-.6pt}%
    {\rule[-\textheight/2]{1ex}{\textheight}}
  }{\textheight}%
}{0.5ex}}%
\stackon[1pt]{#1}{\tmpbox}%
}
\author[A. Kwiatkowska]{Aleksandra Kwiatkowska}
\address{Institut f\"{u}r Mathematische Logik und Grundlagenforschung, Universit\"{a}t  M\"{u}nster,  
Einsteinstrasse 62,
48149  M\"{u}nster,
Germany {\bf{and}} 
Instytut Matematyczny, Uniwersytet Wroc{\l}awski,  pl. Grunwaldzki 2/4, 50-384 Wroc{\l}aw, Poland}
\email{kwiatkoa@uni-muenster.de}

\keywords{universal minimal flows, homeomorphism groups of Wa\.zewski dendrites, Fra\"{i}ss\'{e} limits}
\subjclass[2010]{05D10, 37B05, 54F15, 03C98}
\begin{document}

\maketitle

\begin{abstract}
We study  universal minimal flows of the homeomorphism groups of generalized   Wa\.zewski dendrites $W_P$, $P\subseteq\{3,4,\ldots,\omega\}$.
 If $P$ is finite, we prove that the universal minimal flow of the homeomorphism group $H(W_P)$ is metrizable and we compute it explicitly. This answers a question of  Duchesne.
If $P$ is infinite,
we show that  the universal minimal flow of $H(W_P)$ is not metrizable. 
This provides examples of 
topological groups which are Roelcke precompact and have a non-metrizable universal minimal flow with a comeager orbit.
\end{abstract}

\section{Introduction}

The {\em order} of a point $x$ in a topological space is the number of connected components
we obtain after removing $x$. A {\em ramification point} is a point which has  order at least~3. An {\em endpoint} is a point of order 1. A {\em continuum} is a compact connected topological space.
A {\em dendrite} is a locally connected continuum  that contains no simple closed curve. All dendrites we consider
in this article will be metrizable.
A {\em Wa\.zewski dendrite} $\dom$ is a dendrite such that each ramification point of $\dom$ is of order~$\omega$ and each arc $I$ contained in $\dom$
contains a ramification point.

Moreover, for every $P\subseteq\{3,4,\ldots,\omega\}$, there exists a {\em generalized Wa\.zewski dendrite} $W_P$, that is, a dendrite
 such that each ramification point of $W_P$ is of order that belongs to $P$ and for every $p\in P$ and an arc $I$ contained in $W_P$, $I$
contains a  ramification point of order $p$. For every  $P\subseteq\{3,4,\ldots,\omega\}$, a  generalized 
Wa\.zewski dendrite 
 is unique up to homeomorphism, see Charatonik-Dilks \cite[Theorem 6.2]{CD}.
 Duchesne-Monod \cite{DM} studied structural properties of  homeomorphism groups of generalized 
Wa\.zewski dendrites, in particular, they showed that these groups are simple.

The homeomorphism group of a generalized Wa\.zewski dendrite is isomorphic (as a topological group)
to the automorphism group of a certain Fra\"{i}ss\'{e}-HP structure (i.e. the Fra\"{i}ss\'{e} limit of a family of finite 
 first-order structures, which has the joint embedding and the amalgamation properties, but not necessarily the hereditary property), which we now describe. 
Let $P$ be fixed and consider $W_P$. Let $M_P$ be the set of all ramification points of $W_P$. Let $\lan_P$ be the 
first-order language that consists
of a  4-ary relation symbol $D$ and of  unary relation symbols $K_p$ for every $p\in P$.
We let $\mpp$ to be the structure with universe~$M_P$, $D^{\mpp}(a,b,c,d)$ iff the path in $W_P$
 connecting $a$ and $b$ and the path connecting $c$ and~$d$ do not intersect
 (we emphasize that we allow here trivial paths, i.e. we allow $a=b$ or $c=d$), 
 and let $K_p^{\mpp}(a)$ iff  $a$  is  a ramification point of the order equal to $p$.

Instead of coding the tree structure using the $D$ relation, we could use the ternary betweenness relation
$B$, where $B(a,b,c)$ iff $b$ belongs to the path $ac$. Indeed, $B(a,b,c)$ holds iff $D(a,c,b,b)$ does
not hold, and $D(a,b,c,d)$ does not hold iff there exists $e$ such that $B(a,e,b)$ and $B(c,e,d)$.
Later on, we will also work with a $C$ relation, which will be defined using the $D$ relation,
moreover, the $D$ relation is used to describe boron trees, see \cite{J},
therefore we decided to work in this article with the $D$ relation rather than with the $B$ 
relation.

Propositions 2.4 and 6.1 in \cite{DM} imply:
\begin{proposition}\label{unispo}
The homeomorphism group of the generalized Wa\.zewski dendrite $W_P$, equipped with the uniform metric, is isomorphic (as a topological group)
to the automorphism group of $M_P$, equipped with the pointwise convergence metric.
\end{proposition}

A {\em  tree} is an acyclic connected   undirected graph.
For a tree $T$ we denote by $V(T)$ the set of vertices and by $E(T)$  the set of edges of $T$.
A {\em degree } of a vertex  in a graph is the number of edges that come out of that vertex. An {\em endpoint} is a vertex of degree~1.
A~{\em path} is a tree such that each vertex either is an endpoint or it has degree 2.
Note that for any two vertices in a tree there is exactly one path joining them.
A path joining vertices $a$ and $b$ we will often denote by $ab$.
A {\em rooted tree} is a tree with a  distinguished point, which we call the {\em root}.
On a rooted tree $T$ with the root $r$ we consider 
the tree order $\leq_{T}$ letting $x\leq_{T}y$ iff  $x$ belongs to the path $ry$. 
A {\em branch} in a rooted tree is a path $ra$, where $r$ is the root and $a$ is an endpoint.
The {\em meet } of $a,b\in T$ is the greatest lower bound of $a$ and $b$ with respect to $\leq_{T}$.
 In a rooted tree we can talk about the {\em height} of each vertex.
The root has the height equal to 0 and 
the height of $x\in T$ is taken to be the maximum plus 1 of heights of $\{v\in T\colon v<_T x\}$.
The {\em height } of a rooted tree $T$ is the maximum of the heights of all of its vertices,
we denote it by $ht(T)$.
Note that the height of $x\in T$  is equal to the length of the path $rx$,
where the {\em length} of a path is defined to be the number of edges in the path.
A {\em successor} of a vertex $x$ is any point $y\neq x$ such that $x\leq_T y$. A vertex $y$ is an {\em immediate successor} of a vertex $x$ if it is a successor of $x$
and there is no successor  $w\neq y$ of $x$ such that $x\leq_T w\leq_T y $.

Let $\f_P$ be the family of all finite structures in the language $\lan_P$ such that the universe is a finite tree and the degree of each vertex is different from 2.
If $\A\in\f_P$, we let $D^{\A}(a,b,c,d)$ iff the path $ab$ and the path $cd$, do not intersect. 
Take $K_p$ such that for every $a\in\A$ there is exactly one $p\in P$ such that $K^{\A}_p(a)$, and if $K^{\A}_p(a)$ then
the degree of $a$ is not greater than $p$.

A first-order structure $M$ is {\em ultrahomogeneous} with respect to a family of finite substructures $\f$ if for any finite substructures 
$A,B\subseteq M$, $A,B\in\f$, and an isomorphism $p\colon A\to B$, there is an automorphism of $M$
extending $p$. 

 Proposition 6.1 in \cite{DM} together with Proposition \ref{unispo} imply:
\begin{proposition}
For every $P\subseteq\{3,4,\ldots,\omega\}$, the structure $M_P$ is ultrahomogeneous with respect to $\f_P$.
\end{proposition}

The proposition above implies that $\f_P$ has the joint embedding property and the amalgamation property. Note that $\f_P$ does not have the hereditary property.
Moreover, as additionally for every finite subset $X\subseteq\mpp$ there is $A\in\f_P$ such that $X\subseteq A\subseteq \mpp$, we have that $\mpp$
is the Fra\"{i}ss\'{e} limit  of $\f_P$.

\begin{remark}\label{steps}
Let $i\colon (S, D^S)\to (T, D^T)$ be an embedding  of trees $S$ and $T$ in which each vertex has degree $\neq 2$. Then every edge in $S$ is mapped to a path in $T$ and
$T$ is obtained from $S$ in a sequence of the following simple steps:

\noindent 1. Start with a tree $T'$. Pick an edge $[a,b]$ in $T'$. Let $c$ and $d$ be points not in $T'$. Get $S'$ by removing edge $[a,b]$,
and by adding points $c$ and $d$, and edges $[a,c]$, $[c,b]$ and $[c,d]$.

\noindent 2. Start with a tree $T'$. Pick an endpoint $e$ in $T'$. Let $c$ and $d$ be points not in $T'$. Get $S'$ by
 adding points $c$ and $d$, and edges $[e,c]$ and $[e,d]$.
 
 \noindent 3. Start with a tree $T'$. Pick an vertex $v$ in $T'$ which is not an endpoint.
 Let $c$ be a point not in $T'$. Get
$S'$ by adding the point $c$ and the edge $[v, c]$. 

\end{remark}

\begin{remark}\label{edge}
Note that the relation $D$ remembers which pairs of vertices are joined by an edge.
Given a tree $T$ and $a,b\in T$, $a\neq b$. Then there is an edge between $a$ and $b$ iff for every $c\in T$, $c\neq a,b$ we have that $c$ does not belong to the path $ab$ iff
for  every $c\in T$, $c\neq a,b$, $D^T(a,b,c,c)$ holds.

\begin{remark}
Let $T$ be a tree and let $E$ be the set of endpoints of $T$. Then $D^T\restriction E$ on the set $E$ is an example of a $D$-relation, as defined in \cite[Section 22]{AN}. Moreover, $D^T$ on the tree $T$ satisfies (D1)-(D3) in the definition of a $D$-relation, but not (D4).
\end{remark}

\end{remark}

\section{The universal minimal flow - preliminaries}


Our goal is to compute universal minimal flows of the homeomorphism groups $H(W_P)$, equivalently, of the automorphism groups $\aut(M_P)$.

We will work in the framework provided by Kechris-Pestov-Todorcevic. Let us recall relevant definitions and theorems. The presentation  below is essentially
copied from \cite{BK}, Section 3.6. Lemma \ref{preco}, Theorem \ref{iden}, and Corollary \ref{ident} are proved there.

 A topological group $G$ is {\em extremely amenable} if every $G$-flow has a fixed point.
 A~{\em coloring} of a set $X$ is any function $c\colon  X\to \{1,2,\ldots,r\}$, for some $r\geq 2$;
 we say that $Y\subseteq X$ is  {\em $c$-monochromatic} (or just {\em monochromatic}) if $r\rest Y$ is constant.
 
Let $\g$ be a family of finite structures in a language $\lan$. 
For $A,B\in \g$ write $A\leq B$ if $A$ embeds into $B$. 
For $A,B$ in $\g$,  let ${B \choose A}$~denotes the set of all embeddings of $A$ into $B$.
We say that $A\in\g$ is a
{\em Ramsey object} if  for every $B\in \g$ with $A\leq B$ and every   integer $r\geq 2$  there exists $C\in \g$ such that for every coloring $c\colon  {C \choose A} \to\{1,2,\ldots,r\}$ there exists $h\in {C \choose B}$ such that 
$\{ h\circ f\colon  f\in {B \choose A} \}$ is monochromatic. Note that to check that $A$ is a Ramsey object it suffices to check it only for $r=2$.
We say that $\g$ is a {\em Ramsey class} (or that it has {\em Ramsey property}) if every structure in $\g$ is a Ramsey object.

A structure $A\in\g$ is {\em rigid} if it has  trivial automorphism group. 
 
Kechris-Pestov-Todorcevic \cite{KPT} worked with
 Fra\"{i}ss\'{e} families and their ordered Fra\"{i}ss\'{e} expansions, their work  was generalized by  Nguyen Van Th\'e \cite{NVT} to 
  Fra\"{i}ss\'{e} families and to arbitrary relational Fra\"{i}ss\'{e} expansions. The Kechris-Pestov-Todorcevic correspondence remains true for Fra\"{i}ss\'{e}-HP families,
  which was checked by several people, and it appears in \cite{Z}, see also \cite{BK}.

   \begin{theorem}[Kechris-Pestov-Todorcevic \cite{KPT}, see Theorem 5.1 in \cite{Z}] \label{kpt1}
   Let $\g$ be a Fra\"{i}ss\'{e}-HP family,  let $\mathbb{G}$ be its Fra\"{i}ss\'{e} limit, and let $G=\aut(\mathbb{G})$.
Then the following are equivalent: 
 \begin{enumerate}
 \item The group $G$ is extremely  amenable.
 \item The family $\g$ is a Ramsey class and it consists of rigid structures.
 \end{enumerate}
  \end{theorem}

    Let $\g$ be a Fra\"{i}ss\'{e}-HP family in a language $\lan$,  let $\mathbb{G}$ be its Fra\"{i}ss\'{e} limit, and let $G=\aut(\mathbb{G})$.
    Let $\g^*$ be a  Fra\"{i}ss\'{e}-HP family in a language $\lan^*\supseteq \lan$, $\lan^*\setminus \lan$  relational,
     such that the map defined on $\g^*$ and given by $A^*\mapsto A^*\restriction \lan$ is onto $\g$.
     In that case we say that $A^*$ is an {\em expansion} of $A^*\restriction \lan$ and that $A^*\restriction \lan$ 
     is a {\em reduct} of $A^*$, and that $\g^*$ is an {\em expansion} of $\g$.
  Let $\mathbb{G}^*$ be the  Fra\"{i}ss\'{e} limit of $\g^*$, and let $G^*=\aut(\mathbb{G}^*)$. 
  
We say that the   expansion $\g^*$ of $\g$ is {\em reasonable} if 
  for any $A,B\in\g$, an embedding $\alpha\colon  A\to B$ and an expansion $A^*\in\g^*$ of $A$,
there is an expansion $B^*\in\g^*$ of $B$ such that $\alpha\colon  A^*\to B^*$ is an embedding. It is
   {\em precompact} if  
 for every $A\in\g$ there are only finitely many $A^*\in\g^*$ such that $A^*\restriction \lan= A$. 
  We say that $\g^*$ has the {\em expansion property}  relative to $\g$ if for any $A^*\in\g^*$ there is 
  $B\in\g$ such that for any expansion  $B^*\in\g^*$, there is an embedding $\alpha\colon A^*\to B^*$.
The following proposition explains the importance of the notion of reasonability.

\begin{proposition}[\cite{KPT}, \cite{NVT}, see Proposition 5.3 in \cite{Z}]\label{kpt_reas}
The expansion $\g^*$ of $\g$ is  reasonable if and only if $\mathbb{G}^*\rest \lan= \mathbb{G}$.
\end{proposition}

We say that $\mathcal{G^*}$ has the {\em relative HP} (the relative hereditary property) with respect to~$\mathcal{G}$ 
if for every $A, B\in\g$ such that  $A$ is a substructure of  $B$ and for $B^*\in\g^*$, an expansion of $B$,
we have $B^*\restriction A\in \g^*$. This is equivalent to saying that
for any $A\in\g$ and an embedding $i\colon A\to \mathbb{G}$ there is an expansion
$A^*\in\g^*$ of $A$ such that $i\colon A^*\to \mathbb{G}^*$ is an embedding.
The relative HP property is used to show that when an expansion $\mathcal{G}^*$ of $\mathcal{G}$ is precompact, then  $\aut(\G)/\aut(\G^*)$  is precompact in the quotient of the right uniformity, the proof is 
contained in Section 3.6 in \cite{BK}.
  \begin{lemma}\label{preco}
Suppose that  $\mathcal{G}^*$ is a  reasonable precompact expansion of $\mathcal{G}$
 and that the relative HP holds. Then
the right uniform space $\aut(\G)/\aut(\G^*)$ is precompact.
\end{lemma}

   Below  $(\mathbb{G}, \vec{R})$  denotes an expansion of  $\mathbb{G}$ to a structure in $\lan^*$.
 Instead of $(\mathbb{G}, \vec{R})$ we will often  just write $\vec{R}$.   
   
Define 
\begin{equation*}
\begin{split}
X_{\mathbb{G}^*}=&\{ \vec{R}:
{\rm\ for\ every\ } A\in\g, 
{\rm\ and\ an\ embedding\ }
i\colon  A\to \mathbb{G} {\rm\ there \ exists\ } \\ 
& A^*\in\g^*, {\rm such \ that\ } i\colon A^*\to(\mathbb{G},\vec{R})
{\rm \ is\ an\ embedding }\}.
\end{split}
\end{equation*}
The relative HP implies that
the space $X_{\mathbb{G}^*}$   contains $\mathbb{G}^*$.
   
We make $X_{\mathbb{G}^*}$ a topological space by declaring sets
\[ V_{i, A^*}=\{\vec{R}\in X_{\mathbb{G}^*} :  \text{the map }i\colon  A^*\to (\mathbb{G},\vec{R})  \text{  is an embedding} \}, \]
where $i\colon  A\to \mathbb{G}$ is an embedding,  $A^*\in\g^*$, and $A^*\rest \lan= A$, to be  open.
The group $\aut(\G^*)$ acts continuously on $X_{\mathbb{G}^*}$ via
\[ g\cdot \vec{R}(\bar{a})=\vec{R}(g^{-1}(\bar{a})).\]
Reasonability and  precompactness  of the expansion $\g^*$ of $\g$ imply that the space $X_{\mathbb{G}^*}$ is compact and
zero-dimensional. 

{\bf{From now on till the end of this section,}} we will assume that the expansion $\g^*$ of $\g$ is  reasonable, precompact, and satisfies the relative HP.

\begin{theorem}[\cite{KPT}, \cite{NVT}, see Proposition 5.5 in \cite{Z}]\label{kpt_minim}
 The following are equivalent:
 \begin{enumerate}
 \item The flow $G\acts X_{\mathbb{G}^*} $ is minimal.
 \item  The  family $\g^*$ has  the expansion property relative to $\g$.
 \end{enumerate}
\end{theorem}
   
  \begin{theorem}[Kechris-Pestov-Todorcevic \cite{KPT}, Nguyen Van Th\'e \cite{NVT}, see Theorem 5.7 in \cite{Z}]\label{kpt2}
The following are equivalent:
 \begin{enumerate}
 \item The flow $G\acts X_{\mathbb{G}^*} $ is the universal minimal flow of $G$.
 \item The  family $\g^*$ is a rigid Ramsey class and  has the expansion property relative to~$\g$.
 \end{enumerate}
  \end{theorem}
  
 A proof of Theorem \ref{iden} is contained   in Section 3.6 in \cite{BK}.

Let $\vec{R}^\mathbb{G}$ be such that $\mathbb{G}^*=(\mathbb{G}, \vec{R}^\mathbb{G})$.
 \begin{theorem}\label{iden}
  The map $g\aut(\G^*)\mapsto g\cdot \vec{R}^\mathbb{G}$ from $\aut(\G)/\aut(\G^*)$ to $X_{\mathbb{G}^*}$ is a uniform isomorphism.
  \end{theorem}
  
We will say that flows $G\acts X$ and $G\acts Y$ are {\em isomorphic} if there is a homeomorphism from $X$ onto $Y$ which is a $G$-map.
  \begin{corollary}\label{ident}
  The flow $G\acts \reallywidehat{\aut(\G)/\aut(\G^*)}$ is isomorphic to the flow $G\acts X_{\mathbb{G}^*} $.
  \end{corollary}

\section{The universal minimal flow - construction}
In this section, we show:
\begin{theorem}\label{main}

For any $P\subseteq\{3,\ldots,\omega\}$ there is a reasonable  Fra\"{i}ss\'{e}-HP expansion $\f^*_P$ of $\f_P$, which has the relative HP, the expansion, and the Ramsey properties. In the case when $P$ is finite, this expansion $\f^*_P$ is also precompact.
\end{theorem}
Then using the Kechris-Pestov-Todorcevic correspondence and Proposition \ref{unispo},  we  obtain a description of the universal minimal flow of the homeomorphism group of the generalized Wa\.zewski dendrite $W_P$, for all finite $P$. In particular, we will obtain that this universal minimal flow is metrizable, when $P$ is finite.
This answers a question of Duchesne asked
during his talk
at the Workshop ``Structure and Geometry of Polish groups'' in
Oaxaca in 06/2017.
In the special case, when $P=\{\omega\}$, the universal minimal flow of $H(W_{\{\omega\}})$, independently  of our work, was identified by Duchesne in~\cite{D2}.

Given $\f_P$ in the language $\lan_P$, we  first construct a family $\T^*_P$ of rooted trees with ordered and labeled branches, and then we construct 
the required family $\f^*_P$ that is Ramsey and has the expansion property with respect to $\f_P$.

{\bf{The family $\mathbf{\T^*_P}$.}}
Take $\lan_{\T^*_P}=\lan_P\cup \{ C, \prec, G_1, G_2\ldots\}$, where $C$ is a ternary relation symbol, $\prec, G_1,G_2,\ldots$
are binary relation symbols. If $P$ is finite, it suffices to take $G_1,\ldots, G_{m-1}$, where
 $m=\max(P\setminus\{\omega\})$.

Let $A\in\f_P$.

\noindent{\bf{Step 1:}} Choosing the root for $A$.

Let $x$ be an edge of $A$ or a vertex of $A$ such that its degree is strictly less than $p$ satisfying $K_p^A(x)$.
In the case when $x$ is a vertex, denote $r=x$ and consider $A$  with the distinguished point $r$, which we call the root. Denote this rooted tree by $T_{A,r}$.
In the case when $x=[a,b]$ is an edge, remove $x$ from $A$, take a new point $r$ and add edges $[a,r]$ and $[r,b]$. The obtained tree with the distinguished 
point $r$, which we call the root, denote  by $T_{A,r}$ as before.

For simplicity, write $T=T_{A,r}$. Similarly as before, we let for $a,b,c,d\in T$, $D^T(a,b,c,d)$ iff the paths $ab$ and $cd$ do not intersect.
Let  for $a,b,c\in T$, $C^T(a,b,c)$ iff $D^T(a,b,c,r)$. 
(The relation $C^T$ ``remembers'' that the root $r$ of $T$ is the smallest with respect to $\leq_T$ element of $T$.)

It is crucial that we are allowed to choose  the root both with respect to edges and with respect to vertices.
Otherwise, the relative HP would fail, see Remark \ref{relhp}.

\begin{remark}
Let $T$ be a tree and let $E$ be the set of endpoints of $T$. Then $C^T\restriction E$ is an example of a $C$-relation on the set $E$, as defined in \cite[Section 10]{AN}. Moreover, $C^T$ on the tree $T$ satisfies (C1)-(C3) in the definition of a $C$-relation, but not (C4).
\end{remark}

\noindent{\bf{Step 2:}} Labeling the root  $r$.

If $r\in A$ (which is exactly in the case when in Step 1 the $x$ we picked was a vertex) then already there is $p\in P$ such that $K^A_p(r)$, i.e. $K^T_p(r)$.
Otherwise, if $r\notin A$, we pick some $p\in P$ and let $K^T_p(r)$.

\noindent{\bf{Step 3:}} Ordering and labeling branches of $T$.

Here we have to do two things: we will 
introduce a binary relation that induces an 
order of branches of $T$, and then for every $a\in T$ such that for a finite $p\in P$ we have $K_p^T(a)$, we will
put additional labels on the successors of $a$.

The binary relation $\prec^T$:  For every $a\in T$ we fix a strict linear order $\prec^T_a$ of  its immediate successors.
Then we let $c\prec^T d$ iff for some $a\in T$
there are $i<j$ such that $a_i\leq_T c$ and $a_j\leq_T d$, where $a_1\prec^T_a\ldots\prec^T_a a_n$ are  immediate successors  of $a$, for some $n$.

The binary relations $G^T_i$:  If $a\in T$ and $p\in P\setminus\{\omega\}$ are such that $K^T_p(a)$, and $a_1\prec^T\ldots\prec^T a_n$ are the immediate successors of $a$,
fix an increasing  injection $k\colon \{1,\ldots, n\}\to\{1,\ldots,p-1\}$. We let for $b\in T$,
$G_{k(i)}^T(a,b)$ iff $a_i\leq_T b$.

Clearly $\prec^T$ induces an ordering of branches of $T$. Moreover, if $\omega\notin P$, then $\prec^T$ can be recovered from $G_1^T, G_2^T,\ldots$.
Note that if $n<p-1$ in the definition of an injection~$k$,  $G_1^T, G_2^T,\ldots$
 carry more information that just $\prec^T$.
The reason why we include  $\prec^T$ rather than just work with $G_1^T,G_2^T,\ldots$ is that in the case when $\omega\in P$ and $P$ is finite, we do not want to work with infinitely many $G_i$'s
(otherwise precompactness will fail); in the case $\omega\notin P$, it suffices to work only with $G_1^T, G_2^T,\ldots$
and not introduce $\prec^T$.

Finally, put into $\T^*_P$ any structure obtained from $A$ (in a very non-unique way) in the procedure described in Steps 1-3. Note that every vertex in a $T^*\in \T^*_P$, except possibly the root,
has the degree different from 2.

{\bf{The family $\mathbf{\f^*_P}$.}}
Take $\lan^*_P=\lan_{\T^*_P}\cup \{ R_p\}_{p\in P}\cup \{H_{ij}\}_{1\leq i<j}$, where each $R_p$ and $H_{ij}$ is a binary relation symbol, and $i,j\in\mathbb{N}$.

Start with $A\in\f_P$ and let $T^*\in\T^*_P$  be any rooted tree obtained from $A$. The universes of $A$ and of $T^*$ either  are equal or there is 
an extra point, the root $r$ of $T^*$, which is not in $A$. All the relations in $\lan_{\T^*_P}$ we simply restrict from $T^*$ to $A$. However,
note that in the case $r\notin A$, we "forgot" this way for which $p\in P$ it holds $K^{T^*}_p(r)$ and for which $1\leq i $ it holds 
$G_i^{T^*}(r,a)$, whenever $a\in T^*$, $a\neq r$. In order to remember these two pieces of information after removing the root, we set
for any two incomparable with respect to~$\leq^{T^*}$ elements $a,b\in A$ and $c$ equal to the meet of $a$ and $b$ in the rooted tree~$T^*$:
$R_p^{A}(a,b)$ iff $K_p^{T^*}(c)$ and we set $H_{ij}^A(a,b)$ iff $G_i^{T^*}(c,a)$ and  $G_j^{T^*}(c,b)$.

\begin{proposition}\label{embeddi}
Let $f\colon S\to T$ be an injection between  finite rooted trees $S$ and $T$ with roots $r_S$ and $r_T$, respectively. Then the following are equivalent:
\begin{enumerate}
\item $f$ preserves the relations $C$ (defined with respect to $r_S$ and $r_T$) and $D$;
\item $f$  preserves the relation $C$;
\item $f$ preserves the meet (i.e. for each $a,b\in S$ and their meet $c$, $f(c)$ is the meet of $f(a)$ and $f(b)$).
\end{enumerate}
\begin{proof}
Clearly (1) implies (2). Assume now (2). First notice that then $f$ preserves $\leq_S$ and $\leq_T$. Then note that
if for some $a,b\in S$ and $c$, the meet of $a$ and $b$, we had that $f(c)$ is strictly lower with respect to $\leq_T$
than the meet of $f(a)$ and $f(b)$, then $\neg C^S(a,b,c)$ and $C^T(f(a),f(b),f(c))$, which is impossible. Therefore
we get (3). Now if we assume (3), then $f$ also preserves  $\leq_S$ and $\leq_T$. Essentially from the 
definitions of the relations $C$ and $D$ it follows that if (3) holds then $f$ preserves $C$ and $D$, and hence
we get (1).

\end{proof}


\end{proposition}

\begin{proposition}
The category $\f^*_P$ with embeddings and the category $\T^*_P$ with embeddings  are equivalent via a covariant functor.
\end{proposition}
\begin{proof}
To $A^*\in \f^*_P$ assign $T^*\in \T^*_P$ by adding the root if it is not already in $A^*$. 
From the relation $C^{A^*}$ we can recover where the root is.
Recover the information needed about the root
using the relations $R_p$ and $H_{ij}$. To $T^*\in \T^*_P$ assign $A^*\in \f^*_P$ by removing the root if it was added (which is exactly when the degree
of the root is equal to 2).

To an embedding $f\colon A^*\to B^*$ assign an embedding $g\colon S^*\to T^*$, where $S^*$ corresponds to $A^*$ and $T^*$ corresponds to $B^*$
in the following way. If $S^*$ contains the root $r$ which is not already in $A^*$ and this root was added with respect to an edge $x=[a,b]$, we take $g$ to be the 
extension of $f$ in which  $r$ is mapped to the meet of $f(a)$ and $f(b)$. Again the relations $R_p$ and $H_{ij}$ remember all the information needed
for such a $g$ to be an embedding. 
On the other hand, having an embedding $g\colon S^*\to T^*$, we obtain  an embedding $f\colon A^*\to B^*$  by simply removing the root from $S^*$,
in case it was added, and restricting $g$.
\end{proof}

 Let $\T_P$
denote the set of reducts of elements in $\T^*_P$ to the language $\lan_P\cup\{C\}$. We now prove that the 
family $\f^*_P$ has all the  properties required in Theorem \ref{main}.

\subsection{$\f^*_P$ is reasonable} Let $A, B\in\f_P$  such that  $A$ is a substructure of  $B$, be given, and fix an expansion $A^*\in\f^*_P$ of $A$.
We will define $B^*\in\f^*_P$, an expansion of $B$ which when restricted to $A$ is equal to $A^*$.
If the root $r=r_{A^*}$ of $A^*$ is a vertex such that there is no $b\in B\setminus A$ and an edge $[r,b]$ in  $B$,
 we let $r$ to be the root of $B$. If $r$ is a vertex such that there is $b_0\in B\setminus A$ and an edge $[r,b_0]$ in  $B$, we let the vertex of $B$
 to be any endpoint $e$ of $B$ such that $b_0$ belongs to the path $er$ in $B$.
 The resulted rooted tree denote by $T$ and note that $T\in \T_P$.
  If the root of $A^*$ was added with respect to an edge $[x_1,x_2]$ in $A$,
 then take any edge $[y_1,y_2]\subseteq [x_1,x_2]$ in $B$,  and add the root to $B$ with respect to $[y_1,y_2]$.
 Take $T\in\T_P$ equal to $B$ with the root $r_T$ added  with respect to $[y_1,y_2]$
 and let for every $p\in P$, $K_p^T(r_T)$ iff $K_p^{S^*}(r_{S^*})$, where $S^*\in \T^*_P$ with the root $r_{S^*}$ corresponds to $A^*$.
 
View $S\in\T_P$ equal to the reduct of $S^*$ as embedded in $T$.
 In the case when the root $r_{S^*}$ is not in $A$, this embedding takes $r_{S^*}$ to $r_T$.
 We still have to define $\prec^T$ and $G^T_1,G^T_2\ldots$, which extend  $\prec^{S^*}$ and 
 $G^{S^*}_1,G^{S^*}_2\ldots$.
 For this, for any $b\in T$ and its immediate successors $b_1,\ldots, b_n$, it is enough to define $\prec^T$ on  $\{b_1,\ldots, b_n\}$ and
 specify for each $i$ and $k$ whether $G^T_i(b,b_k)$ holds or not. 
Let $p\in P$ be such that $K_p^B(b)$. 
 In case $b\notin S^*$, we define  $\prec^T$ and $G^T_1,\ldots, G^T_{p-1}$
 in an arbitrary way that Step 3 in the construction allows us. In the case $b\in S^*$, we define  $\prec^T$ and $G^T_1,\ldots, G^T_{p-1}$
 in any way allowed in Step 3 so that additionally if for some $c\in S^*$, $b_k\leq_T c$ and $G_i^{S^*}(b,c)$ then $G_i^T(b, b_k)$ and if  for some $c,d\in S^*$, 
 $b_k\leq_T c$,  $b_l\leq_T d$ and $c\prec^{S^*} d$
 then $b_k\prec^{T} b_l$.
 This defines $T^*\in\T^*_P$, which corresponds to $B^*\in\f^*_P$ we are looking for.

\subsection{$\f^*_P$ is precompact with respect to $\f_P$} Clear. $P$ has to be finite.

\subsection{$\f^*_P$ has the JEP}
For this we can  instead work with the family $\T^*_P$. Take $S^*, T^*\in\T^*_P$. Let $r_{S^*}$ be the root of $S^*$, and let $r_{T^*}$ be the root of $T^*$.
Pick a new element $r$, pick $p\in P$, and if $p<\omega$ pick $1\leq i<j\leq p-1$. Let $R^*\in\T^*_P$ be obtained as follows. We take the union of $S^*$ and $T^*$
together with the point $r$ and vertices $[r,r_{S^*}]$ and $[r,r_{T^*}]$. We declare $r$ to be the root of $R^*$, i.e. we define
$C^{R^*}(a,b,c)$ iff $D^{R^*}(a,b,c,r) $, and let $K_p^{R^*}(r)$. For any $r_{S^*}\leq_{S^*} a$ and  $r_{T^*}\leq_{T^*} b$, let $a\prec^{R^*} b$.
If $p<\omega$, then if  $r_{S^*}\leq_{S^*} a$, we let $G_i^{R^*}(r,a)$ and 
if  $r_{T^*}\leq_{T^*} a$, we let $G_j^{R^*}(r,a)$. 
We also make sure that the degrees of $r_{S^*}$ and $r_{T^*}$ in $R^*$ are at least 3 by adding
additional edges and extending $\prec^{R^*} $ and $G_i^{R^*}$, if needed.
Then this $R^*$ is as required for $S^*$ and $T^*$.
(Note that in this proof we used that the degree of $r_{S^*}$ in $S^*$ is strictly less than $p_0\in P$ such that $K_{p_0}^{S^*}(r_{S^*})$, and similarly for $T^*$.)

\subsection{$\f^*_P$ has the AP} One can show it directly, but it also follows from the rigidity 
of each $A\in \f^*_P$ together with the JEP and the Ramsey properties for 
$\f^*_P$. The proof of this fact is essentially due to  Ne\v{s}et\v{r}il-R\"odl
(see \cite[p. 294, Lemma 1]{NR}), the framework in which they work is somewhat different from ours.
Their proof is  for families of structures which are rigid, hereditary, have the JEP and Ramsey properties, see also
 \cite[p. 129]{KPT}. 
Nevertheless, for a  Fra\"{i}ss\'{e}-HP family $\f$, 
whenever $A\in\f$ then   every structure isomorphic to $A$ is also in $\f$. Therefore the proof presented by
Kechris-Pestov-Todorcevic \cite{KPT} applies to
Fra\"{i}ss\'{e}-HP families as well.

\subsection{$\f^*_P$  has the relative HP}  

Fix $A, B\in\f_P$ and  $B^*\in\f^*_P$ extending $B$.
Take $T^*\in\T^*_P$ that corresponds to $B^*$ and view $A$ as embedded in the $\lan_P$-reduct of $T^*$.  There are either one or two minimal elements in $A\subseteq (T^*, \leq_{T^*})$. Let $r$ be this minimal element, if there is exactly one, and otherwise let $r$
be the meet of the two minimal elements. Take $S=A\cup\{r\}$, a rooted tree with the root $r$.
Let $S^*$ be the substructure of $T^*$ such that the universe is $S$. Then $S^*\in\T^*_P$ and hence the corresponding structure $A^*\in\f^*_P$ satisfies  $B^*\restriction A\in \f^*_P$.

\begin{remark}\label{relhp}
It is possible to have $A, B\in\f_P$, $A$ embedded into $B$, the  expansion $B^*\in\f^*_P$ of $B$,
such that its root was added with respect a vertex, but the root of $A^*$,
the restriction of $B^*$ to $A$,  has the root  added with respect to an edge. 
Let for example $B^*$ be the rooted tree that consists of 4 vertices: $r, a,b, c$, where $r$ is the root, 
and edges $[r,a]$, $[r,b]$, $[r,c]$, and let $A$ be the subtree that consists of 2 vertices $a,b$
and the edge $[a,b]$.
Similarly, it is not hard to give an example of $B^*$ and $A$ such that
$B^*$ has the root added with respect to an edge and $A^*$ has the root added with respect 
to a vertex.
\end{remark}

\subsection{$\f^*_P$  has the expansion property with respect to $\f_P$}
Let $A^*\in\f^*_P$ be given.  Without loss of generality, let the root of $A^*$ belong to $A^*$ (as we can always embed the $A^*$ we started with 
in an element of $\f^*_P$ with such a property). Therefore we can think that $A^*\in\T^*_P$. 
Take a rooted tree $T\in\T_P$
which has the property that all its expansions  to an element in $\T^*_P$
are isomorphic, the degree of the root is $\geq 2$, and $A^*$ embeds in some/every expansion of $T$.
(For this, note that any tree $V\in\T_P$ with the properties: (1) if $x$ and $y$ have the same height, $K_p^V(x)$ and $K_q^V(y)$ hold, then $p=q$;
(2) if $x$ is not an endpoint, $p<\omega$, and $K^V_p(x)$, then $x$ has exactly $p-1$ immediate successors; (3) there is $M\geq 2$ such that every $x$
which is not an endpoint and $K_{\omega}^V(x)$, has exactly $M$ immediate successors; is such that all its expansions  to an element in $\T^*_P$
are isomorphic.)
Finally, let $B$ be obtained as follows. Take $T'$ and $T''$, two disjoint copies of $T$. Denote their roots by $r_{T'}$ and $r_{T''}$, respectively.
The disjoint union of $T'$ and $T''$ together with the edge $[r_{T'}, r_{T''}]$ is a required $B$.
This is because whenever we expand $B$ to a $B^*\in\f^*_P$ then we can embed $T^*$
(the unique expansion of $T$) into $B^*$. 
If, say, the vertex or edge with respect to which is added the root of $B^*$ lies in $T'$, then the unique expansion of $T''$ embeds into $B^*$.

\subsection{$\f^*_P$  has the Ramsey property} 
We generalize the Ramsey theorems by Deuber \cite{D} and by Soki\'c \cite{S} (Theorem 2.2 and Theorem 6.1). For related Ramsey theorems, where it is additionally assumed that endpoints of a rooted tree are mapped to endpoints, see \cite{J}, \cite{BP}, and \cite{Sol}.
\begin{theorem}\label{ramsey}
For any non-empty $P\subseteq\{3,\ldots,\omega\}$, the family $\T^*_P$, and hence the family $\f^*_P$, is Ramsey.
\end{theorem}

Consider $T\in\T^*_P$ with the root $r_T$ and let $q$ be such that $K^T_q(r_T)$. 
Let $V\in\T^*_P$ 
and let  $M$ be the maximum of 2 and 
the number of immediate successors of all vertices in $V$ labeled with $\omega$.
We are going to define  $V[T]\in T^*_P$. First consider $V'\in\T^*_P$ defined as follows.
For every endpoint $e\in V$ take $p_e$ such that $K^V_{p_e}(e)$ and take new points $x_1^e, \ldots, x_{p'_e}^e$,
 where $p'_e=p_e-1$  if $p_e<\omega$ and $p'_e=M$ if $p_e=\omega$, and add edges $[e, x_i^e]$, $i=1,\ldots, p'_e$.
 Then let $V'\in \T^*_P$ be the tree we obtain by letting $K^{V'}_q(x^e_i)$ for each endpoint $e$, and each $i$,
 and (uniquely) choosing  $\prec^{V'}$ and
 $G_i^{V'}$. 
To obtain $V[T]$, to each endpoint of $V'$ attach the tree $T$ by identifying this endpoint with the root $r_T$. 

\begin{example}
Let $V=2^{\leq 1}$ with $K_5^V(\emptyset)$, $K_3^V(0)$, and $K_\omega^V(1)$. 
Let $T=2^{\leq 1}$ with $K_7^T(\emptyset)$, $K_{10}^T(0)$, and $K_6^T(1)$. 
Then $S=V[T]=2^{\leq 3}$ with  $K_5^S(\emptyset)$, $K_3^S(0)$,  $K_\omega^S(1)$,
$K_7^S(00)$, $K_7^S(01)$, $K_7^S(10)$, $K_7^S(11)$, $K_{10}^S(000)$, 
$K_{10}^S(010)$, $K_{10}^S(100)$,  $K_{10}^S(110)$, 
$K_{6}^S(001)$, $K_{6}^S(011)$, $K_{6}^S(101)$, and $K_{6}^S(111)$. 

\end{example}

For a family $\g$ of first-order structures in some language denote by $\g_m$
the family $\{(A_1,\ldots, A_m)\colon A_j\in\g\}$. We say that $(A_1,\ldots, A_m)$ embeds into
$(B_1,\ldots, B_m)$ if for every $j$, $A_j$ embeds into $B_j$.

In the inductive step of  the proof of Theorem \ref{ramsey} we will be using the product Ramsey theorem.
\begin{theorem}[Soki\'c, Theorem 2 in \cite{S2}]
Let $\g$ be a family of first-order structures in some language, which is a Ramsey class.
For any $(A_1,\ldots, A_m), (B_1,\ldots,B_m)\in\g_m$ 
such that $(A_1,\ldots, A_m)$ embeds into $(B_1,\ldots,B_m)$ there is $C\in\g$ such that
for any coloring of embeddings  of $(A_1,\ldots, A_m)$ in $(C,\ldots, C)$ into finitely many colors there is a embedding $h=(h_1,\ldots,h_m)$ of $(B_1,\ldots,B_m)$ in $(C,\ldots, C)$
 such that the set of all functions $h\circ f$, where $f=(f_1,\ldots, f_m)$ is an embedding of $(A_1,\ldots, A_m)$  into $(B_1,\ldots,B_m)$, is monochromatic.
 \end{theorem}
 
 Moreover, from the proof of Soki\'c's theorem it follows that for every $i$:  If  every $A\in\g$ 
 with $ht(A)\leq i$ is a Ramsey object in $\g$, then every $(A_1,\ldots, A_m)\in\g_m$
 with each $A_j$ satisfying $ht(A_j)\leq i$, is a Ramsey object in $\g_m$.

We show that $\T^*_P$ is a Ramsey class , i.e. we show that for every
 $S,T\in \T^*_P$ with $S\leq T$ there exists $U\in \T^*_P$ such that for every coloring $c: {U \choose S} \to\{\mathrm{blue}, \mathrm{red}\}$ there exists $h\in {U \choose T}$ such that 
$\{ h\circ f: f\in {T \choose S} \}$ is monochromatic.

\begin{proof}[Proof of Theorem \ref{ramsey}]
We show that every $S\in\T^*_P$ is a  Ramsey object by induction on the height of $S$.
First let $S\in\T^*_P$ be a one-element structure. Take $T\in\T_P$ such that $S$ embeds into $T$. 
Without loss of generality, for any $x\in T$, which is not an endpoint, and $p\in P$ such that $K^T_p(x)$, if $p<\omega$, then the number of immediate successors of $x$ is equal to $p-1$. Moreover assume that there is $M$ such that the for any $x$, which is not an endpoint, such that $K^T_{\omega}(x)$, the number of immediate 
successors of $x$ is exactly $M$.
Suppose that $S=\{a\}$ and let $p_S\in P$ be such that 
$K^S_{p_S}(a)$.

Let $T_0=T$ and  $T_{k}=T_{k-1}[T]$, $1\leq k\leq h=ht(T)$, and we claim that $U=T_{h}$ is as required.
Denote the set of copies of $T$ attached to $T_{k-1}$ in the construction of $T_k$ by $\T_k$, and let $\T_0=\{T_0\}$.
Color embeddings of $S$ into $U$ into two colors: blue and red.
If there is $k$ and $T'\in\T_k$ such that all embeddings of $S$ into $T'$ are in the same color, we are done.
Otherwise, for each $k$ and $T'\in\T_k$ there is a blue embedding of $S$ into $T'$.

We construct the required embedding $f$ of $T$ into $U$ by induction. First we construct $f(r_T)$, where  $r_T$ denotes the root of $T$.
Let $p_T\in P$ be such that $K^T_{p_T}(r_T)$. If $p_T\neq p_S$, let $f(r_T)=r_U$, where $r_U$ is the root of $U$.
If $p_T=p_S$, let $f(r_T)$ be an image of any  blue embedding of $S$
 into $T'=T_0\in \T_0$.
Now let $x\in T$ be of height $k$ and suppose that we constructed  $f(x)$ and that $f(x)\in T'$ for some $T'\in \T_k$
 Let $x_1\prec^T\ldots\prec^T x_m$
be the list of immediate successors of $x$, and we construct $f(x_1),\ldots, f(x_m)$, each will be in a copy of $T$ that lies in $\T_{k+1}$.
 Suppose that $p\in P$ is such that
$K_{p}^T(x)$ and that $y_1\prec^U\ldots \prec^U y_{p'}$ is the list of immediate successors of $f(x)$, where $p'=p-1$
 if $p<\omega$ and $p'=M$ when $p=\omega$.
 By the construction of $U$, there are $r_1\prec^U\ldots \prec^U r_{p'}$,
such that $r_l$ is a successor of $y_l$ in $U$ and $r_l$ is the root of some $T^l\in\T_{k+1}$.
For each $l$, let $p_l\in P$ be such that $K^T_{p_l}(x_l)$, and if $p_l\neq p_S$, let $f(x_l)$ be equal to the point in $T_l$ corresponding to $x_l$
in the obvious isomorphism between $T$ and $T^l$. Otherwise, if $p_l= p_S$, we let $f(x_l)$ to be the image of a blue embedding of $S$ into  $T^l\subseteq U$.
This gives a ``blue'' embedding of $T$ into $U$ and
 finishes the base step of the induction.

For the inductive step,
let  $S,T\in\T^*_P$ such that $S\leq T$ be given. We assume that  every  tree in $\T^*_P$ of the height strictly less than the height of $S$ is a Ramsey object.
Let $V\in\T^*_P$ be such that whenever we color embeddings of $\{r_S\}$ into $V$ into two colors,
then there exists an embedding $g\colon T\to V$  such that 
$\{ g\circ f\colon f\in {T \choose S} \}$ is monochromatic. 
Without loss of generality, we assume that  for any $x\in V$, which is not an endpoint, and $p\in P$ such that $K^V_p(x)$, if $p<\omega$, then the number of immediate successors of $x$ is equal to $p-1$. 
Let $a_1\prec^S\ldots \prec^S a_k$ be the list of immediate successors of $r_S$, and let $S_i=S_{a_i}=\{b\in S\colon a_i\leq_S b\}$.

Using the well-founded recursion along $V$, for each $x\in V$ we construct a tree $V^x\in\T^*_P$. The  $U=V^{r_V}$ 
 will be as needed for $S$ and $T$ and two colors.
For an endpoint $x\in V$, let $V^x=\{x\}$ with $K_p^{V^x}(x)$ iff $K_p^V(x)$ for every $p\in P$. 
Now let $x\in V$ not be an endpoint, let $x_1\prec^V\ldots\prec^V x_n$ be the list of all immediate successors of $x$, and assume that we already defined
 $V^{x_1},\ldots, V^{x_n}$. Let $V^x_0$ be obtained from the disjoint union of $\{x\}$ and $V^{x_1},\ldots, V^{x_n}$,  adding edges $[x, r_{V^{x_i}}]$.
 For $a\in V^{x_i}$ and $t=1,\ldots,n$ we let
$G^{V^x_0}_t(x,a)$ iff $G^V_t(x, x_i)$. Similarly, for $a\in V^{x_i}$ and $b\in V^{x_j}$ we let $a\prec^{V^x_0} b$ iff $x_i\prec^V x_j$ and let $K_p^{V^x}(x)$ iff $K_p^V(x)$.

If $S$ does not embed into $V^x_0$ in a way that $r_S$ is mapped to $r_{V^x_0}$, 
the root of $V^x_0$, let $V^x=V^x_0$.
 Otherwise, and if $p$ such that $K^V_p(x)$ is finite, apply the product Ramsey theorem to $(S_1,\ldots, S_k)$ and $(V^{x_{b(1)}},\ldots, V^{x_{b(k)}})$,
 where $b\colon\{1,\ldots,k\}\to \{1,\ldots,n\}$ is an increasing injection such that for any $t$ and $i$, $G_t^S(r_S,a_i)$ iff $G_t^V(x,x_{b(i)})$,
 and let $U^x\in\T^*_P$ be the structure we obtain from the product Ramsey theorem. 
 For each $j\in\rm{rng}(b)$, take any $U^{j}\in \T^*_P$ such that $U^x$ embeds into it and for any $p\in P$, $K^{U^{j}}_p(r_{U^{j}})$ iff
 $K^{V^{x_j}}_p(r_{V^{x_j}})$. Finally, let $V^x$ be equal to $V^x_0$ with each $V^{x_{b(i)}}$ replaced by $U^{b(i)}$.
 If $K^V_{\omega}(x)$, 
take $l$ such that whenever we color increasing injections of  $\{1,\ldots, k\}$ to $\{1,\ldots, l\}$ into two colors then there is an increasing injection $g\colon\{1,\ldots, n\}\to \{1,\ldots, l\}$ such that all maps $g\circ f$, where $f\colon \{1,\ldots, k\}\to \{1,\ldots, n\}$  is an increasing injection, 
  are in the same color.
  Let $U_0^x\in\T^*_P$ be any structure that all $V^{x_1},\ldots, V^{x_n}$ embed into it. 
  Enumerate increasing injections of  $\{1,\ldots, k\}$ to $\{1,\ldots, l\}$ into $e_1,\ldots, e_m$.
Define recursively  $U_{i+1}^x$, $i=0,\ldots, m-1$ to be  the result of applying the product Ramsey theorem to $(S_1,\ldots, S_k)$ and $(U_i^x,\ldots, U_i^x)$.
Define   $V^x$ to be  the disjoint union of $\{x\}$ and $l$ many $U^x_m$,
we add edges $[x, r_{U^x_m}]$, and specify $\prec^{U^x_m}$, and let $K_p^{V^x}(x)$ iff $K_p^V(x)$.

Observe that for every $x\in V$ with  immediate successors  $x_1\prec^{V}\ldots\prec^{V} x_n$, which is not an endpoint, we have  
 for every $p\in P$,  $K_p^V(x)$ iff $K_p^{V^x}(r_{V^x})$. If $p$ is finite and such that  $K_p^V(x)$ 
or if $S$ does not embed into $V^x$ in a way that $r_S$ is mapped to $r_{V^x}$, then $x$ has exactly $n$ many immediate successors
 $x'_1\prec^{V^x}\ldots\prec^{V^x} x'_n$ in $V^x$ and they are such that for any $t$ and $i$, $G^{V^x}_t(r_{V^x}, x'_i)$ iff $G^{V}_t(x,x_i)$,
 and for any $p$ and $i$, $K_p^{V^x}(x'_i)$ iff $K_p^V(x_i)$. Moreover, for any coloring into two colors of embeddings of $S$ into $V^x$
 such that  $r_S$ is mapped to $r_{V^x}$, there is an embedding $g\colon V^x_0\to V^x$ taking $r_{V^x_0}$ to $r_{V^x}$ such that
 $\{ g\circ f\colon f\in {V^x_0 \choose S} \text{ taking $r_S$ to } r_{V^x_0} \}$ is monochromatic. 
 If $K_{\omega}^V(x)$, set $V^i=\{a\in V^x_0\colon x_i\leq_{V^x_0} a\}$, where $x_1\prec^{V^x_0}\ldots\prec^{V^x_0} x_n$ are the immediate successors of $x$ in $V^x_0$.
 Then for any coloring into two colors of embeddings of $S$ into $V^x$ such that  $r_S$ is mapped to $r_{V^x}$,
 there are immediate successors  $y_1\prec^{V^x}\ldots\prec^{V^x} y_n$ of $r_{V^x}$ and 
 an embedding $g\colon V^x_0\to V^x$ taking $r_{V^x_0}$ to $r_{V^x}$ and satisfying $y_i\leq_{V^x} g(V^i)$, $i=1,\ldots, n$,
 such that
 $\{ g\circ f\colon f\in {V^x_0 \choose S} \text{ taking $r_S$ to } r_{V^x_0} \}$ is monochromatic.

 Color embeddings  of $S$ into $U$ into two colors. Using the observations above, find an embedding $h\colon V\to U$ such that any two 
 embeddings  $g_1,g_2\colon S\to U$ whose image is contained in $h(V)$ and with $g_1(r_S)=g_2(r_S)$,
 are in the same color. Finally, by the choice of $V$, for the induced coloring of
 embeddings of $\{r_S\}$ into $V$ into two colors,
there exists an embedding $g\colon T\to V$ such that
 $\{ g\circ f\colon f\in {T \choose S}\}$ is monochromatic.
 Then $h\circ g$ is as required.

\end{proof}

We finish this section relating Theorem \ref{ramsey} to the work of Soki\'c  \cite{S}.

A {\em semilattice} is a poset such that every 2 elements have an infimum. If $A$ is a semilattice, we define a binary operation $\circ$ on $A$
by $a\circ b=\inf(a,b)$ and a partial order $\leq_A$ by $a\leq_A b$ iff $a\circ b=a$.
Say that $(A,\circ^A)$ is a {\em treeable semilattice} if  the induced poset is a rooted tree, i.e. it has the minimum, called the root,
and for each $a\in A$, the set $\{b\in A\colon b\leq_A a\}$ is linearly ordered by $\leq_A$.

Let $\T$ be the family of all finite treeable semilattices in the language $\{\circ\}$.
Let $A\in\T$ and say that $\preceq^A$ is a {\em convex ordering} on $A$ if for every $a,b,c\in A$ with $a\circ b=c$, $a\neq c$ and $b\neq c$,
we have $a\preceq^A b$ iff $a'\preceq^A b'$, where $a',b'$ are immediate successors of $c$, $a'\leq_A a$ and $b'\leq_A b$. Denote the set of all convex
ordering on $A$ by $co(A)$ and let 
\[ \mathcal{CT}=\{(A,\circ^A,\preceq^A)\colon (A,\circ^A)\in \T, \preceq^A\in co(A)\}.\]

\begin{theorem}[Soki\'c, Theorem 2.2 in \cite{S}]
$\mathcal{CT}$ is a Ramsey class.
\end{theorem}

The theorem above is a special case of Theorem \ref{ramsey} and is equivalent to the statement that
 $\T^*_{\{\omega\}}$ is a Ramsey class. Indeed,
the categories $\mathcal{CT}$ and $\T^*_{\{\omega\}}$ are equivalent via a covariant functor, which follows from
Proposition \ref{embeddi} and an observation that  convex orderings on  treeable semilattices
correspond to binary relations allowed in Step 3 of the definition of $\T^*_P$.

Similarly, Theorem 6.1 in \cite{S} is equivalent to the statement that each $T^*_{\{k\}}$ is a Ramsey class, $k\geq 3$, therefore again it is a special case of Theorem \ref{ramsey}.

\section{The generalized Wa\.zewski dendrite $W_P$, for an infinite $P\subseteq\{3,4,\ldots,\omega\}$}

In this section, we show that in the case $P$ is infinite, the universal minimal flow of the homeomorphism group of the generalized Wa\.zewski dendrite $W_P$ is non-metrizable, and
we point out two important consequences this fact (see Sections 4.1 and 4.2).


Let $\g$ be a family of finite structures.
Say that $A\in\g$ has {\em Ramsey degree} $\geq t$ iff there exist $A\leq B$ such that for every 
$B\leq C$ there exists a coloring
$c_0\colon {C \choose A} \to\{1,2,\ldots,t\}$ such that for every $g\in {C \choose B}$,
$\{ g\circ f: f\in {B \choose A} \}$ assumes $\geq t$ colors.

The Ramsey degree is infinite if for every $t$ it is $\geq t$.

\begin{theorem}[Zucker \cite{Z}, Theorem 8.7]\label{andy2}
Let $\g$ be a Fra\"{i}ss\'{e}-HP family and let $\G$ be its Fra\"{i}ss\'{e} limit. Then some $A\in\g$ has  infinite Ramsey degree iff the universal minimal flow of $\aut(\G)$ is non-metrizable.
\end{theorem}

Theorems \ref{andy2} and \ref{infinite2} imply that when $P$ is infinite then the universal minimal flow of
$H(W_P)$ is non-metrizable.

\begin{theorem}\label{infinite2}
Suppose that $P$ is infinite. Then there is $A\in\f_P$ which has infinite Ramsey degree.
\end{theorem}
 \begin{proof}
 Let $p_1<p_2<\ldots$ be the increasing enumeration of $P\setminus\{\omega\}$, let $A=\{a,b\}$ be such that $V(A)=\{a,b\}$, $E(A)=\{[a,b]\}$,
 $K^A_{p_1}(a)$ and $K^A_{p_1}(b)$, and let $t\geq 2$ be given. Take $B$  constructed as follows. First let $B_1$ consists of a single point $x_0$
 such that $K^{B_1}_{p_1}(x_0)$. Let $B_2$ be the tree that consists of vertices $x_0, y_1,\ldots, y_{p_1}$ and edges $[x_0, y_i]$,
  $K^{B_2}_{p_1}(x_0)$, and $K^{B_2}_{p_2}(y_i)$, $i=1,\ldots, p_1$. Having constructed $B_k$, $k\leq t$, such that for every endpoint $e$ in $B_k$
  it holds  $K^{B_k}_{p_k}(e)$, we obtain $B_{k+1}$ from $B_k$ in the following way. For every endpoint $e$ in $B_k$ pick new points $y^e_1,\ldots, y^e_{p_k-1}$
  and add vertices $[e, y^e_i]$. Note that $e$ has degree $p_k$ in $B_{k+1}$.
   If $k<t$, we let $K^{B_{k+1}}_{p_{k+1}}(y^e_i)$, and if $k=t$, let $K^{B_{k+1}}_{p_1}(y^e_i)$. Take $B=B_t$.
  
  Now take any $C\in\f_P$ such that $B\leq C$. Pick an endpoint $r$ in $C$ and consider $C$ as a rooted tree with the root $r$.
  Let $c_0\colon {C \choose A} \to\{1,2,\ldots,t\}$ be the following coloring. For an embedding $f\colon A\to C$, if $f(a) $ and $f(b)$ do not lie on the same branch in the rooted tree $C$,
  let $c_0(f)=i$ iff $K^C_{p_i}(c)$, where $c$ is the meet of $f(a)$ and $f(b)$. Otherwise, if $f(a) $ and $f(b)$  do lie on the same branch, let 
  $c_0(f)$ be an arbitrary color from $\{1,\ldots, t\}$. 
  
  Let $g\colon B\to C$ be an embedding. There are $j_1$ and $j_2$ such that $g(x_0)\leq_C g(y_{j_1}), g(y_{j_2})$, $y_{j_1}, y_{j_2} \in B_2$
  (in fact, all $j=1,\ldots,p_1$ except one have this property). Clearly in the rooted tree $B'$,  obtained from $B$ by 
  removing all vertices $z$ such that some $y_j\in B_2$, $j\neq j_1, j_2$,  is on the path connecting $z$ and $x_0\in B_1$, we have that
  for any $i$ there are endpoints $e_1$ and $e_2$ in $B'$ such that the meet of $e_1$ and $e_2$ is a vertex $c$ such that $K^{B'}_{p_i}(c)$.
  That implies that $g(B')$ and hence $g(B)$ assumes all $t$ colors.
  \end{proof}
  
  \begin{corollary}\label{abcd}
  Suppose that $P$ is infinite. Then the universal minimal flow of $H(W_P)$ is non-metrizable.
  \end{corollary}
  
  \subsection{$H(W_P)$ has a  non-metrizable universal minimal flow and is Roelcke precompact}

A subgroup $H$ of $S(X)$, the group of all permutations of a countable set $X$ with the pointwise convergence topology,  is {\em oligomorpic} when for every $n$, the diagonal action of $H$ on $X^n$ has only finitely many
orbits. Note that we do not assume that $H$ is a closed subgroup of $S(X)$.
A topological group $H$ is {\em Roelcke precompact } if for every open neighbourhood $U$ of $1\in H$
there exists a finite set $F\subseteq H$ such that $H=UFU$. As shown by Tsankov \cite[Theorem 2.4]{T}
a subgroup of $S(X)$ is Roelcke precompact if and only if it is an inverse limit of oligomorphic groups.

As observed by Todor Tsankov (private communication in 2013), ${\rm{Aut}}(M_P)$, for each $P\subseteq\{3,4,\ldots,\omega\}$,
is a Roelcke precompact group. This is because when we take 
\[M_n=\{m\in M_P\colon K_p(m) \text{ for some  }  p\in\{3,\ldots,n,\omega \}\}, \] 
\[G_n={\rm{Aut}}(M_n)={\rm{Aut}}(M_n, D^{M_n}, (K_p^{M_n})_{p\in P\cap\{3,\ldots,n,\omega\}}),\]  
and \[H_n=\{h\in G_n\colon \text{there exists } f\in {\rm{Aut}}(M_P) \text{ such that } h=f\restriction M_n\}, \]
then $H_n$ is an oligomorphic group  and the inverse limit of $H_n$ is equal to ${\rm{Aut}}(M_P)$.

Melleray-Nguyen Van Th\'e-Tsankov \cite{MNT} asked:
\begin{question}[Question 1.5 in  \cite{MNT}]\label{que}

 Is the universal minimal flow of every Roelcke precompact Polish group metrizable?

\end{question}
Moreover, Bodirsky-Pinsker-Tsankov \cite{BPT} asked if every  $\omega$-categorical structure has an  
$\omega$-categorical expansion which is Ramsey (which by the work of Zucker \cite{Z} is equivalent
to the question above with ``Roelcke precompact Polish'' replaced by ``oligomorphic'').

Evans-Hubi\v{c}ka-Ne\v{s}et\v{r}il  \cite{EHN} answered Question \ref{que} in the negative. They provided an example of an oligomorphic
  group with  a non-metrizable universal minimal flow. Their example is much more involved than ours,
  it is based on a very non-trivial construction due to Hrushovski~\cite{Hr}, see also \cite{En}.


 \subsection{$H(W_P)$ has a  non-metrizable universal minimal flow with a comeager orbit}

Ben Yaacov-Melleray-Tsankov  \cite{BMT}, generalizing a result of Zucker \cite{Z}, showed:
\begin{theorem}[Theorem 1.2 in \cite{BMT}]
Let $G$ be a Polish group whose universal minimal flow $M(G)$ is metrizable. Then $M(G)$ has a comeager orbit.
\end{theorem}
They asked if the converse holds:
\begin{question}[Question 1.3 in \cite{BMT}]\label{que2}
Suppose that $G$ is a Polish group such that $M(G)$ has a comeager
orbit. Is it true that $M(G)$ is metrizable?

\end{question}

After this preprint was posted on arXiv, Zucker \cite{Zt} showed:
\begin{theorem}\label{thesis}
Let $\g$ be a Fra\"{i}ss\'{e}-HP family.
Suppose that  there exists a reasonable  Fra\"{i}ss\'{e}-HP expansion $\g^*$ of $\g$, which has the relative HP, the expansion, and the Ramsey properties, but the precompactness fails. Then the universal minimal flow of ${\rm Aut}(\mathbb{G})$, where 
$\mathbb{G}$ is the Fra\"{i}ss\'{e} limit of $\mathcal{G}$, has a comeager orbit.
\end{theorem}

Theorems \ref{main} and  \ref{thesis}, together with Corollary \ref{abcd},  provide the negative answer to Question \ref{que2}.

\section*{Acknowledgements}
I thank Todor Tsankov for our discussions on generalized Wa\.zewski dendrites in 2013.
I also thank Miodrag  Soki\'c, Bruno Duchesne, and the referee,
for many valuable remarks that improved the presentation of the paper.
The author was supported by Narodowe Centrum Nauki grant 2016/23/D/ST1/01097.

\end{document}